\documentclass[a4paper, 12pt]{amsart}

\usepackage[a4paper]{geometry}
\usepackage{amssymb}
\usepackage{amsthm}
\usepackage{amsmath}
\usepackage{graphicx}
\usepackage{hyperref}
\usepackage{multirow}
\usepackage{array}
\usepackage[colorinlistoftodos]{todonotes}
\usepackage[english]{babel}
\usepackage{lmodern}
\usepackage{comment}
\usepackage{bbold}
\usepackage{mathtools}
\usepackage{enumerate}
\usepackage{cite}
\usepackage{lscape}
\usepackage{tikz}

\allowdisplaybreaks

\usepackage[utf8]{inputenc}
\usepackage[T1]{fontenc}

\newtheorem{theorem}{Theorem}[section]

\newtheorem{corollary}[theorem]{Corollary}

\theoremstyle{definition}

\newtheorem{proposition}[theorem]{Proposition}

\theoremstyle{remark}
\newtheorem{remark}[theorem]{Remark}
\newcommand{\longcomment}[1]{}

\DeclareMathOperator{\Spec}{Spec}



\numberwithin{equation}{section}

\author{Bartosz Naskręcki}
\address{Faculty of Mathematics and Computer Science\\
	Adam Mickiewicz University in Poznań\\
	ul. Uniwersytetu Poznańskiego 4 \\
	61-614, Poznań\\ 
	Poland}
\email{bartosz.naskrecki@amu.edu.pl}

\title[Explicit equations of 800 conics on a Barth-Bauer quartic]{Explicit equations of 800 conics on a Barth-Bauer quartic}

\makeatletter
\renewcommand{\tocsection}[3]{%
	\indentlabel{\@ifnotempty{#2}{\bfseries\ignorespaces#1 #2\quad}}\bfseries#3}
\renewcommand{\tocsubsection}[3]{%
	\indentlabel{\@ifnotempty{#2}{\ignorespaces#1 #2\quad}}#3}

\newcommand\@dotsep{4.5}
\def\@tocline#1#2#3#4#5#6#7{\relax
	\ifnum #1>\c@tocdepth 
	\else
	\par \addpenalty\@secpenalty\addvspace{#2}%
	\begingroup \hyphenpenalty\@M
	\@ifempty{#4}{%
		\@tempdima\csname r@tocindent\number#1\endcsname\relax
	}{%
		\@tempdima#4\relax
	}%
	\parindent\z@ \leftskip#3\relax \advance\leftskip\@tempdima\relax
	\rightskip\@pnumwidth plus1em \parfillskip-\@pnumwidth
	#5\leavevmode\hskip-\@tempdima{#6}\nobreak
	\leaders\hbox{$\m@th\mkern \@dotsep mu\hbox{.}\mkern \@dotsep mu$}\hfill
	\nobreak
	\hbox to\@pnumwidth{\@tocpagenum{\ifnum#1=1\bfseries\fi#7}}\par
	\nobreak
	\endgroup
	\fi}
\AtBeginDocument{%
	\expandafter\renewcommand\csname r@tocindent0\endcsname{0pt}
}
\def\l@subsection{\@tocline{2}{0pt}{2.5pc}{5pc}{}}
\makeatother
\begin{document}
	
	\begin{abstract}
		We present equations of $800$ conics on a Mukai's smooth quartic which is an explicit  model of a K3 surface recently considered by Alex Degtyarev in \cite{Alex1,Alex2}. The number $800$ is a current record for the number of irreducible conics on a smooth quartic in ordinary projective space.
	\end{abstract}

	\maketitle

\section{Introduction}
In a recent paper Alex Degtyarev\cite{Alex1} has proved the existence of a smooth quartic surface  $X\subset\mathbb{P}^{3}_{\mathbb{C}}$ which contains $800$ irreducible conics. In the sequel \cite{Alex2} Degtyarev proved that up to projective linear change of coordinates there is exactly one smooth quartic in $\mathbb{P}^3$ which contains $16$  pairwise disjoint conics (it is called a \textit{Barth-Bauer} quartic) which admits 800 conics. The number $800$ is the maximum of the total number of conics on a Barth-Bauer quartic. It is expected that $800$ is the maximum number of conics on \textit{any} smooth quartic in $\mathbb{P}^3$, \cite{Alex2}. A surface described by Degtyarev has N\'{e}ron-Severi group of discriminant $-160$ and transcendental lattice $\langle 4\rangle\oplus \langle 40\rangle$. Xavier Roulleau has pointed to a concrete model of such a surface which was studied previously by Bonnaf\'e and Sarti and constructed by Mukai. The same surface was also studied by Brandhorst and Hashimoto \cite{BH}. Its equation in the projective space $\mathbb{P}^{3}_{z_0,z_1,z_2,z_3}$ is
\[X:\sum_{i=0}^{3} z_i^4-6 \sum_{i<j} z_i^2 z_j^2=0.\]
This quartic is smooth over $\mathbb{C}$ and admits an action of a group $P_{G_{Mu}}$, which is an extension of the Mathieu group $M_{20}$ by $\mathbb{Z}/2$, and is a projectivization of a complex reflection group $G_{29}$ in Shephard- Todd classification \cite{ST}. The group $G_{29}$ is generated by four matrices $s_1,s_2,s_3,s_4$ ($i=\sqrt{-1}$)

$$
\begin{array}{cc}
	s_1=\begin{pmatrix}
		1  &0 & 0 & 0\\
		0  &1 & 0 & 0\\
		0 & 0 & 1 & 0\\
		0 & 0 & 0 &-1\\
	\end{pmatrix}, &
	s_2={\frac{1}{2}}\begin{pmatrix}
		1&  1&  i&  i\\
		1 & 1& -i& -i\\
		-i&  i&  1& -1\\
		-i  &i &-1&  1\\
	\end{pmatrix},\\
	&\\
	s_3=\begin{pmatrix}
		0 &1& 0& 0\\
		1 &0 &0 &0\\
		0 &0 &1& 0\\
		0& 0 &0& 1\\
	\end{pmatrix},&
	s_4=\begin{pmatrix}
		1 &0& 0& 0\\
		0 &0 &1 &0\\
		0& 1 &0 &0\\
		0 &0 &0 &1\\
	\end{pmatrix}. \\
\end{array}
$$

In \cite{BS} Bonnaf\'e and Sarti  established the equations of 320 conics lying on that surface, located on two orbits upon action of the automorphism group $P_{G_{Mu}}$.

We find the remaining equations of irreducible conics on the surface $X$ and explain the heuristic argument behind this discovery. We also present a set of conics which constitute a basis of the N\'{e}ron-Severi group and a collection of $16$ pairwise disjoint conics. 

Moreover, in Section \ref{sec:grobner_more} we prove, independently from the work 
\cite{Alex1, Alex2},  that on the surface X there is exactly 800 conics. Our proof uses only elaborate Gr\"{o}bner bases calculations, avoiding the group structure on X and the properties of the N\'{e}ron-Severi group.

Let $\mathcal{C}$ denote the set of irreducible conics on $X$. Let $C_1$, $C_2$ be the two conics found in \cite{BS}.

\begin{align*}
	C_1 &: z_0+z_1+z_2=0,\quad z_1^2 + z_1z_2 + z_2^2 +\frac{3+\sqrt{10}}{2} z_3^2=0,\\
	C_2 &:  z_0+z_1+z_2=0,\quad z_1^2 + z_1z_2 + z_2^2 +\frac{3-\sqrt{10}}{2} z_3^2=0.
\end{align*}

Let $K=\mathbb{Q}(i=\sqrt{-1},\sqrt{2},\sqrt{5})$ and let $A=\frac{(\sqrt{5}+1)i}{2}$. Let $C_3$ be a conic

\[C_3: z_2+A z_3=0,\quad z_0^2+2 \sqrt{2} z_1 z_0+z_1^2+\frac{3 (\sqrt{5}+1)}{2}  z_3^2=0.\]

The way in which the equation of the conic $C_3$ was found is explained in Section \ref{sec:heuristic}.
\begin{theorem}\label{prop:conics}
The set $\{C_1,C_2,C_3\}$ generates the set $\mathcal{C}$ of $800$ irreducible conics on $X$ under the action of the group $P_{G_{Mu}}$. The corresponding orbits have length $160,160$ and $480$.
\end{theorem}
\begin{proof}
A conic $C_1$ is stabilized by the group $\textrm{Stab}_{P_{G_{Mu}}}(C_1)\subset P_{G_{Mu}}$ which is generated by $(z_0:z_1:z_2:z_3)\mapsto (z_2:z_0:z_1:-z_3)$ and $(z_0:z_1:z_2:z_3)\mapsto (z_1:z_0:z_2:-z_3)$. Group $\textrm{Stab}_{P_{G_{Mu}}}(C_1)$ is of order $12$. Hence by the orbit-stabilizer theorem the length of the orbit $\Omega_1$ of $C_1$ under the action of $P_{G_{Mu}}$ equals $1920/12=160$.

By a direct inspection we find that $C_2$ does not belong to the orbit of $C_1$. A stabilizer group $\textrm{Stab}_{P_{G_{Mu}}}(C_2)$ of $C_2$ equals $\textrm{Stab}_{P_{G_{Mu}}}(C_1)$. Thus the orbit $\Omega_2$ of $C_2$ has length $160$.

Finally, the curve $C_3$ is stabilized by the subgroup $\textrm{Stab}_{P_{G_{Mu}}}(C_3)$ of order $4$ and generated by $(z_0:z_1:z_2:z_3)\mapsto (i z_0:i z_1:-i z_2:-i z_3)$, $(z_0:z_1:z_2:z_3)\mapsto (z_1:z_0:z_2:z_3)$. So the orbit $\Omega_3$ of $C_3$ has length $480$, proving that its disjoint from $\Omega_1$ and $\Omega_2$. Hence the union of three orbits described above has $800$ elements.
\end{proof}

\section{A heuristic behind Theorem \ref{prop:conics}}\label{sec:heuristic}

One can notice by direct inspection that all conics in the orbits $\Omega_1$ and $\Omega_2$ have the cutting hyperplane equation which has three or four non-zero coefficients. It is then natural to check whether there are any conics on $X$ which are generated by the hyperplanes with less non-zero coefficients. Without loss of generality (due to symmetry of the coordinates) we have essentially one cutting hyperplane with one non-zero coefficient $H:z_3=0$. It is easy to check that the curve $H\cap X$ is a smooth quartic curve, hence by the Pl\"{u}cker formula it has geometric genus $3$.

What remains to be checked is the reducibility of the quartics generated by the cuts $H_{\alpha}\cap X$ where the hyperplane has the form $H_{\alpha}: z_2+\alpha z_3=0$.  

To numerically confirm our expectation that some conics in the complement of $\Omega_1\cup\Omega_2$ belonged to the set $H_{\alpha}\cap X$ for suitable choices of complex numbers $\alpha$ we have studied this problem for the reduction of $X$ to the finite field $\mathbb{F}_{p^2}$ for certain primes $p$ (for example $p=41$ is the least split prime in $K$). It turned out that for several primes there was a pair of conics generated by the plane $z_2+\alpha z_3=0$ for a suitable choice of $\alpha\in\mathbb{F}_{p^2}$. Lead by this hunch we considered the following setup.

Let $C_\alpha: \sum_{i=0}^{3} z_i^4-6 \sum_{i<j} z_i^2 z_j^2=0, z_2+\alpha z_3=0$ be a family of curves parametrized by $\alpha\in\mathbb{C}$.

\begin{proposition}\label{prop:sing_quartics}
Over the field of rational functions $\mathbb{C}(t)$ the curve $C_t$ is geometrically irreducible, smooth and of genus $3$.
\end{proposition}
\begin{proof}
Elimination of the variable $z_2$ leads to a planar model in $\mathbb{P}^2$ over $\mathbb{C}(t)$ in variables $z_0,z_1,z_3$. It is a routine calculation to verify that the curve $C_t$ over $\mathbb{C}(t)$ is a smooth quartic (hence of genus $3$).
\end{proof}
\begin{remark}
Hence, due to Proposition \ref{prop:sing_quartics} there is a finite subset of numbers $\alpha\in\mathbb{C}$ (neccesarily algebraic) for which the curves $C_{\alpha}$ are not smooth. As will turn out in the sequel there is 12 such singular quartics, out of which only 4 are geometrically irreducible. That will produce $16$ new conics which were not previously found in \cite{BS}.
\end{remark}

Let $\pi:\mathcal{C}\rightarrow\mathbb{A}^{1}$ be a flat morphism whose fibre $\pi^{-1}(\alpha)$ above $\alpha$ is a curve $C_{\alpha}$. Let $\mathcal{S}\subset \mathbb{A}^{1}$ be a finite scheme whose closed points correspond to the fibres for which the curve is not smooth.
\begin{proposition}\label{prop:degeneration}
The scheme $\mathcal{S}$ equals $\Spec\mathbb{C}[x]/(g(x))$ where the polynomial $g(x)$ is
\[g(x)=\underbrace{\left(x^2-2 x-1\right) \left(x^2+2 x-1\right)}_{g_1(x)} \underbrace{\left(x^4+3 x^2+1\right) \left(5 x^4+6 x^2+5\right)}_{g_2(x)}.\]
\end{proposition}
\begin{proof}
In the equation $C_\alpha$ we eliminate variable $z_2$ by substitution from the second equation. The resulting quartic polynomial 
\[F(z_0,z_1,z_3)= z_0^4+z_1^4+z_3^4+\alpha^4 z_3^4-6 \left(\alpha^2 z_3^4+\left(\alpha^2+1\right) \left(z_0^2+z_1^2\right) z_3^2+z_0^2 z_1^2\right)\]
defines a quartic curve $F(z_0,z_1,z_3)=0$ in the projective plane $\mathbb{P}^{2}_{z_0,z_1,z_3}$. We compute $F_0=\frac{\partial F}{\partial z_0}$,  $F_1=\frac{\partial F}{\partial z_1}$ and $F_3=\frac{\partial F}{\partial z_3}$ and consider an ideal $I=\langle F_0,F_1,F_3,F\rangle$ in $R=\mathbb{Z}[z_0,z_1,z_3,\alpha]$. It defines a scheme $\Spec R/I$ whose reduced subscheme is the union of the following irreducible components
\begin{align*}
	\mathcal{G}_{1,1}&:  z_0=z_1=\alpha^2+2\alpha+1=0,\\
	\mathcal{G}_{1,2}&:  z_0=z_1=\alpha^2-2\alpha+1=0,\\
	\mathcal{G}_{2,1}&: z_0 =z_1^2 - 3\alpha^2 z_3^2  - 3z_3^2=\alpha^4+3\alpha^2+1=0,\\
	\mathcal{G}_{2,2}&: z_1 =z_0^2 - 3\alpha^2 z_3^2  - 3z_3^2=\alpha^4+3\alpha^2+1=0,\\
	\mathcal{G}_{2,3}&: z_0+z_1=z_1^2 + 3/2 \alpha^2 z_3^2 + 3/2 z_3^2=\alpha^4 + 6/5\alpha^2 + 1=0,\\
	\mathcal{G}_{2,4}&: z_0-z_1=z_1^2 + 3/2 \alpha^2 z_3^2 + 3/2 z_3^2=\alpha^4 + 6/5\alpha^2 + 1=0,\\
    \mathcal{H}&: z_0=z_1=z_3=0.
\end{align*}
All but the last component $\mathcal{H}$ correspond to a choice of $\alpha\in\mathbb{C}$ and a projective point $(z_0:z_1:z_3)$ which is singular on the curve $C_{\alpha}$.
\end{proof}

 \begin{proposition}\label{prop:degen}
 Assume the notation from Proposition \ref{prop:sing_quartics}. If $\alpha\in\mathbb{C}$ is a root of $g_1(x)$, then $C_{\alpha}$ is a singular model of a geometrically irreducible curve of geometric genus $2$.
 	
If $\alpha$ is such that $g_2(\alpha)=0$, then $C_{\alpha}$ is a union of two irreducible conics which are defined over $\mathbb{Q}(i,\sqrt{2},\sqrt{5})$.

The set of $8$ pairs of conics $C_{\alpha}$ such that $g_2(\alpha)=0$ is parametrized by $\alpha\in \{\pm \frac{1}{2}(\sqrt{5}\pm 1), \pm \frac{1}{\sqrt{5}} (1\pm 2i)\}$.
\end{proposition}
\begin{proof}

\textbf{$\mathcal{G}_{1,j}$ components:} If $\alpha$ is a root of $g_1(x)$, then we have a unique nodal singularity at $(0:0:1)$. Hence, it follows from the Pl\"{u}cker formula that the geometric genus of $C_{\alpha}$ is $(4-1)(4-2)/2-1 = 2$. It follows from Bezout theorem that the curve $C_{\alpha}$ is geometrically irreducible.

\textbf{$\mathcal{G}_{2,j}$ components:} When $\alpha$ is a root of $g_2(x)$ we find four nodal singularities on $C_{\alpha}$. Suppose that the curve were geometrically irreducible. Then Pl\"{u}cker formula would imply that the geometric genus of $C_{\alpha}$ equals $-1$, a contradiction. Hence, $C_{\alpha}$ is geometrically reducible. We form a product of quadratic polynomials in $z_0,z_1,z_3$ with unknown coefficients to reveal the factorization in each case. A simple Gr\"{o}bner basis calculation leads to the conclusion that $C_{\alpha}$ is a union of two irreducible conics defined over $K$.
\end{proof}

\section{More Gr\"{o}bner bases computations}\label{sec:grobner_more}
Suppose we consider a general conic $C$ on the surface $X$. It follows that there exists a hyperplane $H$ such that $H\cap X$ is a union $C\cup \widetilde{C}$ of two conics \cite[Lem. 2.2.6]{Deland}. Let $H_{a,b,c,d}: d z_0+c z_1 +b z_2+ a z_3  = 0$ be a general equation of the hyperplane. Finding a projective quadruple $[a:b:c:d]$ which corresponds to the splitting $H\cap X = C \cup \widetilde{C}$ is equivalent to finding all factorizations into quadrics of the polynomial which we obtain under elimination of one of the variables $z_i$ under substitution into the equation defining $X$. So we consider the following cases:
\begin{itemize}
	\item[(i)] $d=1$,
	\item[(ii)] $d = 0$, $c=1$,
	\item[(iii)] $d = 0$, $c=0$, $b=1$,
	\item[(iv)] $d=c=b=0$, $a=1$.
\end{itemize}
Let $f(z_0,z_1,z_2,z_3) = \sum_{i=0}^{3} z_i^4-6 \sum_{i<j} z_i^2 z_j^2$ and 
$$c_{m}(T_0,T_1) = m_1 T_0^2+m_2 T_0+m_3+m_4 T_1^2+m_5 T_1+m_6 T_0 T_1$$ 
where $m=(m_1,m_2,m_3,m_4,m_5,m_6)$. 
Let $A=(a_1,a_2,a_3,a_4,a_5,a_6)$ and $B=(b_1,b_2,b_3,b_4,b_5,b_6)$.

\medskip
Case (iv). Equation $f(z_0,z_1,z_2,0)=0= z_3$ defines a non-singular curve of genus $3$ in $\mathbb{P}^{3}$.

Case (iii). This was considered above and defines $16$ irreducible conics.

Case (ii). Here we have $f(z_0,-az_3-bz_2,z_2,z_3)\mid _{z_3=1} = c_{A}(z_2,z_0)\cdot c_{B}(z_2,z_0)$ which defines a system of equations in variables $(a_),(b_i),a,b$ for which we find the minimal reduced Gr\"{o}bner basis with respect to the lexicographical order within seconds (on a standard laptop). This defines equations of $64$ conics where the final equation in the variable $b$ is
\[(b-1) b (b+1) \left(b^2+4\right) \left(4 b^2+1\right) \left(b^4+3 b^2+1\right) \left(5 b^4+6 b^2+5\right)=0.\]

Case (i). Now, the setup is $f(-az_3-bz_2-cz_1,z_1,z_2,z_3)\mid _{z_3=1} = c_{A}(z_2,z_1)\cdot c_{B}(z_2,z_1)$. Computation of the minimal reduced Gr\"{o}bner basis with respect to $A,B,a,b,c$ (in lexicographical order) was performed on a single Intel Xeon 2.7 GHz and took approximately 12.5h and less than 1GB of RAM memory, using the default Gr\"{o}bner basis implementation in Magma V2.24-4. The basis has 186 equations, the final one in $c$ is
\[\begin{split}
(c-1) c (c+1) \left(c^2+1\right) \left(c^2+4\right) \left(c^2-4 c-1\right) \left(c^2-4 c+5\right) \left(c^2-c-1\right) \left(c^2+c-1\right)\cdot\\ \left(c^2+4 c-1\right) \left(c^2+4 c+5\right) \left(4 c^2+1\right) \left(5 c^2-6 c+5\right) \left(5 c^2-4 c+1\right)\cdot\\
 \left(5 c^2+4 c+1\right) \left(5 c^2+6 c+5\right) \left(c^4+3 c^2+1\right) \left(c^4+18 c^2+1\right) \left(5 c^4-6 c^2+5\right) \left(5 c^4+6 c^2+5\right)=0.
 \end{split}\]
This basis defines 720 conics on X. 
These computations lead to the following corollary which is proved using only standard computer algebra techniques.

\begin{corollary}\label{cor:conics_with_groebner}
There exist exactly 800 conics on the surface X.
\end{corollary}
\begin{proof}
Every conic $C$ on $X$ corresponds to a unique hypersurface $H$ in $\mathbb{P}^3$. The Gr\"{o}bner bases setup (i)--(iii) shows that there are exactly $400$ such hyperplanes, producing in total $800$ irreducible conics.
\end{proof}

\begin{remark}
Notice that this proof uses no group action on the surface $X$ and is independent of the work \cite{Alex1,Alex2}. In particular, this raises a question of whether such a straightforward computational method of enumerating conics can be applied to a general smooth quartic and be used to show that the upper bound on the number of irreducible conics on such surfaces equals $800$. One should realize that the length of the minimal reduced Gr\"{o}bner basis computation expands in time and space exponentially.
\end{remark}

\section*{Acknowledgments}
I would like to thank Alex Degtyarev for the inspiration that gave the impetus for this work. Without his papers I would never have tried to solve this problem. Thanks goes also to Xavier Roulleau who pointed us to a projective model of a quartic which satisfied the right conditions. I would like to thank C\'edric Bonnaf\'e and Alessandra Sarti for finding the first 320 conics on~$X$. This has lowered the amount of my effort by $40\%$. I would like to also thank Dino Festi for discussions about this project. We are also grateful to the University of Bristol for providing us with the access to Magma cluster CREAM.

\section*{Appendix A: code}
This appendix contains a complete Magma code which allows an interested user to reproduce the data and computations described in Proposition \ref{prop:conics} and Proposition \ref{prop:degeneration} and Corollary \ref{cor:conics_with_groebner}.

\medskip
\noindent
\textbf{Magma code which reproduces the data from Proposition \ref{prop:conics}}
\medskip
\begin{verbatim}
Q<sq2,sq5>:=NumberField([Polynomial([-2,0,1]),Polynomial([-5,0,1])]);
R<T>:=PolynomialRing(Q);
pol:=T^2 + 1/2*(sq5 + 3);
Qext<Am>:=ext<Q|pol>;
i:=Sqrt(Qext!(-1));
assert Am eq i/2*(-1-sq5);
A:=-Am;
s1:=Matrix(4,4,[1,0,0,0,0,1,0,0,0,0,1,0,0,0,0,-1]);
s2:=1/2*Matrix(4,4,[1,1,i,i,1,1,-i,-i,-i,i,1,-1,-i,i,-1,1]);
s3:=Matrix(4,4,[0,1,0,0,1,0,0,0,0,0,1,0,0,0,0,1]);
s4:=Matrix(4,4,[1,0,0,0,0,0,1,0,0,1,0,0,0,0,0,1]);
Gmu:=MatrixGroup<4,Qext|[s1,s2,s3,s4]>;
assert Order(Gmu) eq 7680;

P<z0,z1,z2,z3>:=ProjectiveSpace(Qext,3);
f:=z0^4+z1^4+z2^4+z3^4-6*(z0^2*z1^2+z0^2*z2^2
+z0^2*z3^2+z1^2*z2^2+z1^2*z3^2+z2^2*z3^2);
Xmu:=Surface(P,[f]);
v:=Matrix(4,1,[z0,z1,z2,z3]);
function LA(M,tup)
return Eltseq(ChangeRing(M,CoordinateRing(P))*tup);
end function;
Omegaf:=func<C|{Curve(Scheme(P,[Evaluate(x,LA(el,v)): 
	x in DefiningPolynomials(C)] )): el in Gmu}>;
Omegaelf:=func<C,g|Curve(Scheme(P,[Evaluate(x,LA(g,v)): 
x in DefiningPolynomials(C)] ))>;

//known conics
CC:=IrreducibleComponents(Xmu meet Scheme(P,z0+z1+z2));
Cp:=Curve(CC[1]);
Cm:=Curve(CC[2]);
Omegap:={Curve(Scheme(P,[Evaluate(x,LA(el,v)): 
	x in DefiningPolynomials(Cp)] )): el in Gmu};
Omegam:={Curve(Scheme(P,[Evaluate(x,LA(el,v)): 
	x in DefiningPolynomials(Cm)] )): el in Gmu};
assert #Omegap eq 160;
assert #Omegam eq 160;
assert  #(Omegam meet Omegap) eq 0;

//new conics
CC2:=Xmu meet Scheme(P,[z2+A*z3]);
IrrCC2:=IrreducibleComponents(CC2);
CC21:=Curve(IrrCC2[1]);
CC22:=Curve(IrrCC2[2]);
oo1:=Omegaf(CC21);
oo2:=Omegaf(CC22);
assert #oo1 eq 480;
assert oo1 eq oo2;

Conics800:=Omegap join Omegam join oo1;
\end{verbatim}

\medskip
\noindent
\textbf{Magma code which reproduces the data from Proposition \ref{prop:degeneration}}
\medskip

\begin{verbatim}
	R<z0,z1,z3,A>:=AffineSpace(Rationals(),4);
	eq1:=[4*z0*(z0^2 - 3*z1^2 - 3*z3^2 - 3*A^2*z3^2), 
	-4*z1*(3*z0^2 - z1^2 + 3*z3^2 + 3*A^2*z3^2), 
	-4*z3*(3*z0^2 + 3*A^2*z0^2 + 3*z1^2 + 3*A^2*z1^2 - z3^2 + 6*A^2*z3^2 - A^4*z3^2), 
	z0^4 - 6*z0^2*z1^2 + z1^4 - 6*z0^2*z3^2 - 6*A^2*z0^2*z3^2 - 6*z1^2*z3^2 - 
	6*A^2*z1^2*z3^2 + z3^4 - 6*A^2*z3^4 + A^4*z3^4];
	S:=Scheme(R,eq1);
	redS:=ReducedSubscheme(S);
	irrS:=IrreducibleComponents(redS);
	print(irrS);

	Qext<alpha>:=NumberField(Polynomial([1,0,0,0,7,0,0,0,1]));
	irrSK:=[IrreducibleComponents(ChangeRing(el,Qext)): el in irrS];
	print(irrSK);
\end{verbatim}

\medskip
\noindent
\textbf{Magma code which reproduces the data from Corollary \ref{cor:conics_with_groebner}}
\medskip

\begin{verbatim}
//case (i)
R<a1,a2,a3,a4,a5,a6,b1,b2,b3,b4,b5,b6,a,b,c>:=PolynomialRing(Rationals(),15);
eqq:=[-1 + 6*a^2 - a^4 + a3*b3, 12*a*b - 4*a^3*b + a3*b2 + a2*b3,
6 + 6*a^2 + 6*b^2 - 6*a^2*b^2 + a3*b1 + a2*b2 + a1*b3, 
12*a*b - 4*a*b^3 + a2*b1 + a1*b2,
-1 + 6*b^2 - b^4 + a1*b1, a5*b3 + a3*b5 + 12*a*c - 4*a^3*c,
a5*b2 + a6*b3 + a2*b5 + a3*b6 + 12*b*c - 12*a^2*b*c, 
a5*b1 + a6*b2 + a1*b5 + a2*b6 + 12*a*c -
12*a*b^2*c, a6*b1 + a1*b6 + 12*b*c - 4*b^3*c, 
6 + 6*a^2 + a4*b3 + a3*b4 + a5*b5 + 6*c^2 -
6*a^2*c^2, 12*a*b + a4*b2 + a2*b4 + a6*b5 + a5*b6 - 12*a*b*c^2,
6 + 6*b^2 + a4*b1 + a1*b4 + a6*b6 + 6*c^2 - 6*b^2*c^2, 
a5*b4 + a4*b5 + 12*a*c - 4*a*c^3,
a6*b4 + a4*b6 + 12*b*c - 4*b*c^3, -1 + a4*b4 + 6*c^2 - c^4];
time gro3:=GroebnerBasis(eqq);
gr3:=[gro3[i]: i in [1..#gro3]| &and[Degree(gro3[i],R.j) eq 0: j in [1..12]]];
Aff<A,B,C>:=AffineSpace(Rationals(),3);
S3:=Scheme(Aff,[Evaluate(x,[0: i in [1..12]] cat [A,B,C]): x in gr3]);

Z<T>:=PolynomialRing(Rationals());
K<theta>:=SplittingField(T^4+3*T^2+1);
K2<theta2>:=NumberField(T^8 + 7*T^4 + 1);

assert #RationalPoints(S3,K) eq Degree(S3);
assert #RationalPoints(S3,K) eq 360;

SS3:=Scheme(AffineSpace(R),gro3 cat [a4-1]);
bb:=gro3 cat [a4-1];
irrS3:=IrreducibleComponents(S3);

cr:=0;
ratli:=[];
for co in irrS3 do
pp:=DefiningPolynomials(co);
ppev:=[Evaluate(x,[a,b,c]): x in pp];
bbv:=bb cat ppev;
time bbvg:=GroebnerBasis(bbv);
ss:=Scheme(AffineSpace(R),bbvg);
ssK2:=ChangeRing(ss,K2);
rr:=RationalPoints(ssK2);
cr:=cr+#rr;
ratli:=ratli cat [Eltseq(x): x in rr];
print cr;
end for;
assert cr eq 720;
\end{verbatim}

\begin{verbatim}
//case (ii)
R<a1,a2,a3,a4,a5,a6,b1,b2,b3,b4,b5,b6,a,b>:=PolynomialRing(Rationals(),14);
eq2:=[-1 + 6*a^2 - a^4 + a3*b3, 12*a*b - 4*a^3*b + a3*b2 + a2*b3, 
6 + 6*a^2 + 6*b^2 - 6*a^2*b^2 + a3*b1 + a2*b2 + a1*b3, 
12*a*b - 4*a*b^3 + a2*b1 + a1*b2, 
-1 + 6*b^2 - b^4 + a1*b1, a5*b3 + a3*b5, a5*b2 + a6*b3 + a2*b5 + a3*b6, 
a5*b1 + a6*b2 + a1*b5 + a2*b6, a6*b1 + a1*b6, 6 + 6*a^2 + a4*b3 + a3*b4 + a5*b5, 
12*a*b + a4*b2 + a2*b4 + a6*b5 + a5*b6, 6 + 6*b^2 + a4*b1 + a1*b4 + a6*b6, 
a5*b4 + a4*b5, a6*b4 + a4*b6, -1 + a4*b4];
gro2:=GroebnerBasis(eq2);
gr2:=[gro2[i]: i in [1..#gro2]| &and[Degree(gro2[i],R.j) eq 0: j in [1..12]]];
Aff<A,B>:=AffineSpace(Rationals(),2);
S2:=Scheme(Aff,[Evaluate(x,[0: i in [1..12]] cat [A,B]): x in gr2]);

Z<T>:=PolynomialRing(Rationals());
K<theta>:=SplittingField(T^4+3*T^2+1);
K2<theta2>:=NumberField(T^8 + 7*T^4 + 1);

assert #RationalPoints(S2,K) eq Degree(S2);
assert #RationalPoints(S2,K) eq 32;


SS2:=Scheme(AffineSpace(R),gro2 cat [a4-1]);
SS2K2:=ChangeRing(SS2,K2);
rat2:=RationalPoints(SS2K2);
assert #rat2 eq 64;
\end{verbatim}

\begin{verbatim}
//case (iii)
R<a1,a2,a3,a4,a5,a6,b1,b2,b3,b4,b5,b6,a>:=PolynomialRing(Rationals(),13);
eq1:=[-1 + 6*a^2 - a^4 + a3*b3, a3*b2 + a2*b3, 
6 + 6*a^2 + a3*b1 + a2*b2 + a1*b3, a2*b1 + a1*b2, 
-1 + a1*b1, a5*b3 + a3*b5, a5*b2 + a6*b3 + a2*b5 + a3*b6, 
a5*b1 + a6*b2 + a1*b5 + a2*b6, 
a6*b1 + a1*b6, 6 + 6*a^2 + a4*b3 + a3*b4 + a5*b5, a4*b2 + a2*b4 + a6*b5 + a5*b6, 
6 + a4*b1 + a1*b4 + a6*b6, a5*b4 + a4*b5, a6*b4 + a4*b6, -1 + a4*b4];

gro1:=GroebnerBasis(eq1);
gr1:=[gro1[i]: i in [1..#gro1]| &and[Degree(gro1[i],R.j) eq 0: j in [1..12]]];
Aff<A>:=AffineSpace(Rationals(),1);
S1:=Scheme(Aff,[Evaluate(x,[0: i in [1..12]] cat [A]): x in gr1]);

Z<T>:=PolynomialRing(Rationals());
K<theta>:=SplittingField(T^4+3*T^2+1);
K2<theta2>:=NumberField(T^8 + 7*T^4 + 1);

assert #RationalPoints(S1,K) eq Degree(S1);
assert #RationalPoints(S1,K) eq 8;

SS1:=Scheme(AffineSpace(R),gro1 cat [a4-1]);
SS1K2:=ChangeRing(SS1,K2);
rat1:=RationalPoints(SS1K2);
assert #rat1 eq 16;
\end{verbatim}

\section*{Appendix B: Generators of the N\'{e}ron-Severi group}
We list here the equations of $20$ conics on the surface $X$ which span the N\'{e}ron-Severi group, and their Gram matrix. In Figure \ref{fig:NS_gram} we present the Gram matrix of a basis of $20$ conics spanning the N\'{e}ron-Severi group.  Each row in the Figure \ref{fig:bas} consists of two coefficient vectors, representing two equations of a given conic $C: \sum_{i\leq j} a_{i,j} z_{i}z_{j}=0, \sum_{i=0}^{3}b_{i} z_i=0$. The entries are vectors defined over $\mathbb{Q}(i=\sqrt{-1},\sigma_2=\sqrt{2},\sigma_5=\sqrt{5})$ ($\sigma_{10}=\sigma_2\sigma_5, \sigma_{-2}=i\sigma_2,\sigma_{-5}=i\sigma_5,\sigma_{-10}=i\sigma_{10}$). The rows of Figure \ref{fig:bas} correspond to vertices in the adjacency graph on Figure \ref{fig:NS_dual_graph}. Every edge of this graph represents the fact that two curves (which correspond to vertices) have multiplicity one intersection.

\section*{Appendix C: Kummer lattice}
We find a Kummer configuration $\mathcal{K}$ of $16$ disjoint conics on $X$. The conics which form our Kummer sublattice are stable under the action of a subgroup $H$ of order $128$. The group $H$ is spanned by the matrices $g_1=diag(-1,1,1,-1)$, $g_2=diag(i,i,i,i)$,
\[g_3=\left(
\begin{array}{cccc}
	0 & 0 & -1 & 0 \\
	0 & -1 & 0 & 0 \\
	-1 & 0 & 0 & 0 \\
	0 & 0 & 0 & 1 \\
\end{array}
\right),\quad g_4=\frac{i+1}{2}\left(
\begin{array}{cccc}
	0 & 1 & 0 & 1 \\
	-i & 0 & -i & 0 \\
	0 & 1 & 0 & -1 \\
	-i & 0 & i & 0 \\
\end{array}
\right)\]
such that $g_1^2=g_2^2=g_4^2=g_3^4=I$. A subgroup of $H$ which fixes all elements of the set $\mathcal{K}$ is cyclic of order $4$ and generated by $g_2$.
The notation in Figure \ref{fig:kum_bas} follows that in Appendix B.


\begin{figure}[htb]
	\[\arraycolsep=2pt\def\arraystretch{1}
	\left(
	\begin{array}{cccccccccccccccccccc}
		-2 & 0 & 0 & 0 & 0 & 0 & 0 & 0 & 0 & 1 & 0 & 0 & 0 & 1 & 0 & 0 & 0 & 0 & 0 & 0 \\
		0 & -2 & 0 & 0 & 0 & 0 & 0 & 0 & 0 & 0 & 1 & 0 & 0 & 1 & 0 & 0 & 0 & 0 & 0 & 0 \\
		0 & 0 & -2 & 0 & 0 & 0 & 0 & 0 & 0 & 0 & 0 & 0 & 0 & 0 & 0 & 0 & 1 & 0 & 0 & 0 \\
		0 & 0 & 0 & -2 & 0 & 0 & 0 & 0 & 0 & 1 & 1 & 0 & 1 & 0 & 0 & 0 & 0 & 0 & 0 & 0 \\
		0 & 0 & 0 & 0 & -2 & 0 & 0 & 0 & 0 & 0 & 0 & 0 & 0 & 1 & 0 & 0 & 1 & 0 & 0 & 0 \\
		0 & 0 & 0 & 0 & 0 & -2 & 0 & 0 & 0 & 0 & 0 & 0 & 1 & 1 & 0 & 0 & 0 & 0 & 0 & 0 \\
		0 & 0 & 0 & 0 & 0 & 0 & -2 & 0 & 0 & 0 & 0 & 0 & 1 & 0 & 1 & 0 & 0 & 0 & 0 & 0 \\
		0 & 0 & 0 & 0 & 0 & 0 & 0 & -2 & 0 & 0 & 0 & 0 & 0 & 0 & 0 & 1 & 1 & 0 & 0 & 0 \\
		0 & 0 & 0 & 0 & 0 & 0 & 0 & 0 & -2 & 0 & 0 & 0 & 1 & 0 & 0 & 1 & 0 & 0 & 0 & 0 \\
		1 & 0 & 0 & 1 & 0 & 0 & 0 & 0 & 0 & -2 & 0 & 0 & 0 & 0 & 0 & 0 & 0 & 0 & 0 & 1 \\
		0 & 1 & 0 & 1 & 0 & 0 & 0 & 0 & 0 & 0 & -2 & 0 & 0 & 0 & 0 & 0 & 0 & 0 & 1 & 0 \\
		0 & 0 & 0 & 0 & 0 & 0 & 0 & 0 & 0 & 0 & 0 & -2 & 0 & 0 & 0 & 0 & 0 & 0 & 1 & 0 \\
		0 & 0 & 0 & 1 & 0 & 1 & 1 & 0 & 1 & 0 & 0 & 0 & -2 & 0 & 0 & 0 & 0 & 0 & 0 & 0 \\
		1 & 1 & 0 & 0 & 1 & 1 & 0 & 0 & 0 & 0 & 0 & 0 & 0 & -2 & 0 & 0 & 0 & 0 & 0 & 0 \\
		0 & 0 & 0 & 0 & 0 & 0 & 1 & 0 & 0 & 0 & 0 & 0 & 0 & 0 & -2 & 0 & 0 & 0 & 0 & 0 \\
		0 & 0 & 0 & 0 & 0 & 0 & 0 & 1 & 1 & 0 & 0 & 0 & 0 & 0 & 0 & -2 & 0 & 0 & 0 & 0 \\
		0 & 0 & 1 & 0 & 1 & 0 & 0 & 1 & 0 & 0 & 0 & 0 & 0 & 0 & 0 & 0 & -2 & 0 & 0 & 0 \\
		0 & 0 & 0 & 0 & 0 & 0 & 0 & 0 & 0 & 0 & 0 & 0 & 0 & 0 & 0 & 0 & 0 & -2 & 0 & 1 \\
		0 & 0 & 0 & 0 & 0 & 0 & 0 & 0 & 0 & 0 & 1 & 1 & 0 & 0 & 0 & 0 & 0 & 0 & -2 & 0 \\
		0 & 0 & 0 & 0 & 0 & 0 & 0 & 0 & 0 & 1 & 0 & 0 & 0 & 0 & 0 & 0 & 0 & 1 & 0 & -2 \\
	\end{array}
	\right)\]
	\caption{Gram matrix of a particular choice of a basis conics in the N\'{e}ron-Severi group of $X$.}\label{fig:NS_gram}
\end{figure}

\begin{figure}[htb]
	\begin{tikzpicture}[scale=0.9]	
		
		\node(1) at (5,1){1};
		\node(2) at (3,2){2};
		\node(3) at (0,0){3};
		\node(4) at (5,4){4};
		\node(5) at (1,1){5};
		\node(6) at (2,2){6};
		\node(7) at (2,5){7};
		\node(8) at (0,2){8};
		\node(9) at (1,4){9};
		\node(10) at (5,2){10};
		\node(11) at (4,2){11};
		\node(12) at (3,3){12};	
		\node(13) at (2,4){13};
		\node(14) at (2,1){14};
		\node(15) at (2,6){15};
		\node(16) at (0,4){16};
		\node(17) at (0,1){17};
		\node(18) at (7,2){18};
		\node(19) at (4,3){19};
		\node(20) at (6,2){20};
		
		\draw (3)  --(17)--(8)--(16)--(9)--(13)--(4)--(10)--(1)--(14)--(5)--(17);
		\draw (14)--(6)--(13)--(7)--(15);
		\draw (14)--(2)--(11)--(19)--(12);
		\draw (11)--(4);
		\draw (10)--(20)--(18);
		
	\end{tikzpicture}
	\caption{Adjacency graph corresponding to the Gram matrix of a particular choice of a basis conics in the N\'{e}ron-Severi group of $X$.}\label{fig:NS_dual_graph}
\end{figure}
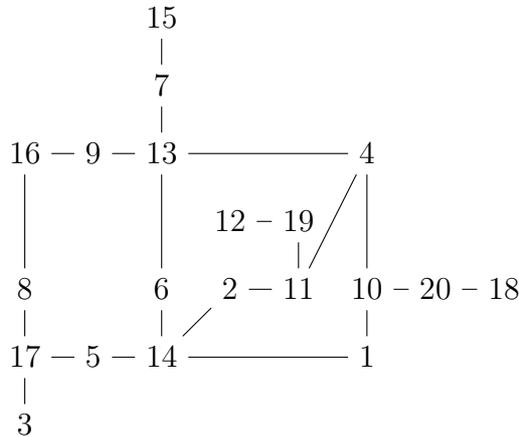

\newpage

{\pagestyle{plain}
	\begin{landscape}
		\tiny
		\begin{figure}
			\[\arraycolsep=2pt\def\arraystretch{2}
			\begin{array}{l||l|l|l|l|l|l|l|l|l|l||l|l|l|l}
				No & a_{0,0} & a_{0,1}& a_{0,2}& a_{0,3}& a_{1,1}& a_{1,2}& a_{1,3}& a_{2,2}& a_{2,3}& a_{3,3} & b_0& b_1& b_2& b_3\\
				\hline\hline
				1& 0 & 0 & 0 & 0 & 1 & 0 & 0 & \frac{i}{3}+\frac{\sigma _{-2}}{3}+\frac{\sigma _2}{6}+\frac{1}{6} & 0 & \frac{i}{3}-\frac{\sigma _{-2}}{3}-\frac{\sigma _2}{6}+\frac{1}{6} & 1 & \frac{\sigma _5}{5}-\frac{2 \sigma _{-5}}{5} & 0 & 0 \\
				2& 0 & 0 & 0 & 0 & 1 & 0 & 0 & \frac{i}{3}-\frac{\sigma _{-2}}{3}-\frac{\sigma _2}{6}+\frac{1}{6} & 0 & \frac{i}{3}+\frac{\sigma _{-2}}{3}+\frac{\sigma _2}{6}+\frac{1}{6} & 1 & \frac{2 \sigma _{-5}}{5}-\frac{\sigma _5}{5} & 0 & 0 \\
				3& 0 & 0 & 0 & 0 & 1 & 0 & -1 & \frac{3}{2}-\frac{\sigma _{10}}{2} & 0 & 1 & 1 & -1 & 0 & 1 \\
				4& 0 & 0 & 0 & 0 & 1 & 0 & 1 & \frac{\sigma _{10}}{2}+\frac{3}{2} & 0 & 1 & 1 & 1 & 0 & 1 \\
				5& 0 & 0 & 0 & 0 & 1 & 0 & 2 i & \frac{\sigma _{10}}{2}+\frac{3}{2} & 0 & \frac{3 \sigma _{10}}{2}+\frac{7}{2} & 1 & -1 & 0 & -2 i \\
				6& 0 & 0 & 0 & 0 & 1 & 0 & -2 i & \frac{3}{2}-\frac{\sigma _{10}}{2} & 0 & \frac{7}{2}-\frac{3 \sigma _{10}}{2} & 1 & 1 & 0 & -2 i \\
				7& 0 & 0 & 0 & 0 & 1 & 1 & 0 & 1 & 0 & \frac{\sigma _{10}}{2}+\frac{3}{2} & 1 & -1 & -1 & 0 \\
				8& 0 & 0 & 0 & 0 & 1 & 2 i & 0 & \frac{7}{2}-\frac{3 \sigma _{10}}{2} & 0 & \frac{3}{2}-\frac{\sigma _{10}}{2} & 1 & -1 & -2 i & 0 \\
				9& 0 & 0 & 0 & 0 & 1 & 2 i & 0 & \frac{3 \sigma _{10}}{2}+\frac{7}{2} & 0 & \frac{\sigma _{10}}{2}+\frac{3}{2} & 1 & 1 & 2 i & 0 \\
				10& 0 & 0 & 0 & 0 & 1 & -2 i-\sigma _{-5}+\frac{\sigma _2}{2} & \frac{\sigma _{-2}}{2}-\sigma _5-2 & \frac{\sigma _{-10}}{2}+\sigma _{-2}-\sigma _5-2 & 5 i+2 \sigma _{-5}+3 \sigma _2+\frac{3 \sigma _{10}}{2} & -\sigma _{-10}-2 \sigma _{-2}+\sigma _5+2 & 1 & -i & -\sigma _5-2 & 2 i+\sigma _{-5} \\
				11& 0 & 0 & 0 & 0 & 1 & -2 i+\sigma _{-5}-\frac{\sigma _2}{2} & -\frac{\sigma _{-2}}{2}+\sigma _5-2 & \frac{\sigma _{-10}}{2}-\sigma _{-2}+\sigma _5-2 & 5 i-2 \sigma _{-5}-3 \sigma _2+\frac{3 \sigma _{10}}{2} & -\sigma _{-10}+2 \sigma _{-2}-\sigma _5+2 & 1 & -i & \sigma _5-2 & 2 i-\sigma _{-5} \\
				12& 0 & 0 & 0 & 0 & 1 & \frac{i}{3}+\frac{\sigma _{-2}}{2}-\frac{2 \sigma _2}{3}+1 & -\frac{\sigma _{-10}}{3}+\frac{\sigma _{-5}}{3}+\frac{\sigma _5}{3}-\frac{\sigma _{10}}{6} & 1 & -\frac{\sigma _{-10}}{3}+\frac{\sigma _{-5}}{3}+\frac{\sigma _5}{3}-\frac{\sigma _{10}}{6} & i-\frac{5 \sigma _{-2}}{6}+\frac{1}{3} & 1 & \frac{2 \sigma _5}{5}-\frac{\sigma _{-5}}{5} & \frac{2 \sigma _5}{5}-\frac{\sigma _{-5}}{5} & 1 \\
				13& 0 & 0 & 0 & 0 & 1 & -\frac{2 i}{3}+\frac{\sigma _{10}}{6}-\frac{\sigma _{-10}}{3}+\frac{1}{3} & \frac{5 i}{3}-\frac{\sigma _{-10}}{6} & -\frac{i}{3}-\frac{\sigma _{-10}}{6} & \frac{2 i}{3}+\frac{\sigma _{-10}}{3}+\frac{\sigma _{10}}{6}+\frac{1}{3} & -1 & 1 & i+2 & 1 & i-2 \\
				14& 0 & 0 & 0 & 0 & 1 & \frac{5 i}{3}+\frac{\sigma _{-10}}{6} & \frac{2 i}{3}-\frac{\sigma _{-10}}{3}-\frac{\sigma _{10}}{6}+\frac{1}{3} & -1 & \frac{2 i}{3}+\frac{\sigma _{10}}{6}-\frac{\sigma _{-10}}{3}-\frac{1}{3} & \frac{i}{3}-\frac{\sigma _{-10}}{6} & 1 & 2-i & i+2 & 1 \\
				15& 0 & 0 & 0 & 0 & 1 & \frac{5 i}{3}+\frac{\sigma _{-10}}{6} & -\frac{2 i}{3}+\frac{\sigma _{-10}}{3}-\frac{\sigma _{10}}{6}+\frac{1}{3} & -1 & \frac{2 i}{3}-\frac{\sigma _{-10}}{3}-\frac{\sigma _{10}}{6}+\frac{1}{3} & \frac{\sigma _{-10}}{6}-\frac{i}{3} & 1 & i+2 & i-2 & 1 \\
				16& 0 & 0 & 0 & 0 & 1 & \frac{\sigma _{-10}}{6}-\frac{5 i}{3} & \frac{2 i}{3}+\frac{\sigma _{-10}}{3}-\frac{\sigma _{10}}{6}-\frac{1}{3} & -1 & \frac{2 i}{3}+\frac{\sigma _{-10}}{3}+\frac{\sigma _{10}}{6}+\frac{1}{3} & -\frac{i}{3}-\frac{\sigma _{-10}}{6} & 1 & -i-2 & i-2 & 1 \\
				17& 0 & 0 & 0 & 0 & 1 & \frac{2 i}{3}+\frac{\sigma _{-10}}{3}-\frac{\sigma _{10}}{6}-\frac{1}{3} & \frac{\sigma _{-10}}{6}-\frac{5 i}{3} & -\frac{i}{3}-\frac{\sigma _{-10}}{6} & \frac{2 i}{3}+\frac{\sigma _{-10}}{3}+\frac{\sigma _{10}}{6}+\frac{1}{3} & -1 & 1 & i+2 & -1 & 2-i \\
				18& 0 & 0 & 0 & 0 & 1 & \frac{i}{3}+\frac{2 \sigma _2}{3}-\frac{\sigma _{-2}}{2}+1 & -\frac{\sigma _{-10}}{3}-\frac{\sigma _{-5}}{3}-\frac{\sigma _5}{3}-\frac{\sigma _{10}}{6} & 1 & -\frac{\sigma _{-10}}{3}-\frac{\sigma _{-5}}{3}-\frac{\sigma _5}{3}-\frac{\sigma _{10}}{6} & i+\frac{5 \sigma _{-2}}{6}+\frac{1}{3} & 1 & \frac{\sigma _{-5}}{5}-\frac{2 \sigma _5}{5} & \frac{\sigma _{-5}}{5}-\frac{2 \sigma _5}{5} & 1 \\
				19& 1 & -2 \sigma _2 & 0 & 0 & 1 & 0 & 0 & 0 & 0 & \frac{3 \sigma _5}{2}+\frac{3}{2} & 0 & 0 & 1 & -\frac{i}{2}-\frac{\sigma _{-5}}{2} \\
				20& 1 & 2 \sigma _2 & 0 & 0 & 1 & 0 & 0 & 0 & 0 & \frac{3}{2}-\frac{3 \sigma _5}{2} & 0 & 0 & 1 & \frac{\sigma _{-5}}{2}-\frac{i}{2} \\
			\end{array}\]
			\caption{Equations of $20$ conics which span the N\'{e}ron-Severi group of $X$.}\label{fig:bas}
		\end{figure}
\end{landscape}}

\newpage
{\pagestyle{plain}
	\begin{landscape}
		\tiny
		\begin{figure}
			\[\arraycolsep=2pt\def\arraystretch{2}
			\begin{array}{l|l|l|l|l|l|l|l|l|l||l|l|l|l}
				a_{0,0} & a_{0,1}& a_{0,2}& a_{0,3}& a_{1,1}& a_{1,2}& a_{1,3}& a_{2,2}& a_{2,3}& a_{3,3} & b_0& b_1& b_2& b_3\\
				\hline\hline
				0 & 0 & 0 & 0 & 1 & 0 & 0 & -\frac{\sigma _5}{6}-\frac{1}{6} & -\frac{s_2}{3}-\frac{\sigma _{10}}{3} & -\frac{\sigma _5}{6}-\frac{1}{6} & 1 & \frac{\sigma _{-5}}{2}-\frac{i}{2} & 0 & 0 \\
				0 & 0 & 0 & 0 & 1 & 0 & 0 & -\frac{\sigma _5}{6}-\frac{1}{6} & \frac{s_2}{3}+\frac{\sigma _{10}}{3} & -\frac{\sigma _5}{6}-\frac{1}{6} & 1 & \frac{i}{2}-\frac{\sigma _{-5}}{2} & 0 & 0 \\
				0 & 0 & 0 & 0 & 1 & -2 s_2 & 0 & 1 & 0 & \frac{3}{2}-\frac{3 \sigma _5}{2} & 1 & 0 & 0 & \frac{\sigma _{-5}}{2}-\frac{i}{2} \\
				0 & 0 & 0 & 0 & 1 & 2 s_2 & 0 & 1 & 0 & \frac{3}{2}-\frac{3 \sigma _5}{2} & 1 & 0 & 0 & \frac{i}{2}-\frac{\sigma _{-5}}{2} \\
				0 & 0 & 0 & 0 & 1 & -\frac{\sigma _{-10}}{3}-\frac{\sigma _{-2}}{3}-\frac{2 \sigma _5}{3}+\frac{4}{3} & -\frac{\sigma _{-10}}{3}+\sigma _{-2}+\frac{\sigma _5}{3}-1 & \frac{\sigma _{-10}}{3}-\frac{\sigma _{-2}}{3}-\frac{\sigma _5}{3}+\frac{1}{3} & -\frac{\sigma _{-10}}{3}+\sigma _{-2}-\frac{2 \sigma _5}{3} & -\frac{1}{3} 2 \sigma _{-2}-\frac{1}{3} & 1 & \frac{\sigma _5}{2}+\frac{1}{2} & -1 & -\frac{\sigma _5}{2}-\frac{1}{2} \\
				0 & 0 & 0 & 0 & 1 & -\frac{\sigma _{-10}}{3}-\frac{\sigma _{-2}}{3}-\frac{2 \sigma _5}{3}+\frac{4}{3} & \frac{\sigma _{-10}}{3}-\sigma _{-2}-\frac{\sigma _5}{3}+1 & \frac{\sigma _{-10}}{3}-\frac{\sigma _{-2}}{3}-\frac{\sigma _5}{3}+\frac{1}{3} & \frac{\sigma _{-10}}{3}-\sigma _{-2}+\frac{2 \sigma _5}{3} & -\frac{1}{3} 2 \sigma _{-2}-\frac{1}{3} & 1 & -\frac{\sigma _5}{2}-\frac{1}{2} & 1 & -\frac{\sigma _5}{2}-\frac{1}{2} \\
				0 & 0 & 0 & 0 & 1 & -\frac{\sigma _{-10}}{3}+\frac{2 \sigma _5}{3}-\frac{\sigma _{-2}}{3}-\frac{4}{3} & -\frac{\sigma _{-10}}{3}+\sigma _{-2}-\frac{\sigma _5}{3}+1 & -\frac{\sigma _{-10}}{3}+\frac{\sigma _{-2}}{3}-\frac{\sigma _5}{3}+\frac{1}{3} & \frac{\sigma _{-10}}{3}-\sigma _{-2}-\frac{2 \sigma _5}{3} & \frac{2 \sigma _{-2}}{3}-\frac{1}{3} & 1 & -\frac{\sigma _5}{2}-\frac{1}{2} & -1 & -\frac{\sigma _5}{2}-\frac{1}{2} \\
				0 & 0 & 0 & 0 & 1 & -\frac{\sigma _{-10}}{3}+\frac{2 \sigma _5}{3}-\frac{\sigma _{-2}}{3}-\frac{4}{3} & \frac{\sigma _{-10}}{3}-\sigma _{-2}+\frac{\sigma _5}{3}-1 & -\frac{\sigma _{-10}}{3}+\frac{\sigma _{-2}}{3}-\frac{\sigma _5}{3}+\frac{1}{3} & -\frac{\sigma _{-10}}{3}+\sigma _{-2}+\frac{2 \sigma _5}{3} & \frac{2 \sigma _{-2}}{3}-\frac{1}{3} & 1 & \frac{\sigma _5}{2}+\frac{1}{2} & 1 & -\frac{\sigma _5}{2}-\frac{1}{2} \\
				0 & 0 & 0 & 0 & 1 & \frac{\sigma _{-10}}{3}+\frac{\sigma _{-2}}{3}-\frac{2 \sigma _5}{3}+\frac{4}{3} & -\frac{\sigma _{-10}}{3}+\sigma _{-2}-\frac{\sigma _5}{3}+1 & -\frac{\sigma _{-10}}{3}+\frac{\sigma _{-2}}{3}-\frac{\sigma _5}{3}+\frac{1}{3} & -\frac{\sigma _{-10}}{3}+\sigma _{-2}+\frac{2 \sigma _5}{3} & \frac{2 \sigma _{-2}}{3}-\frac{1}{3} & 1 & \frac{\sigma _5}{2}+\frac{1}{2} & -1 & \frac{\sigma _5}{2}+\frac{1}{2} \\
				0 & 0 & 0 & 0 & 1 & \frac{\sigma _{-10}}{3}+\frac{\sigma _{-2}}{3}-\frac{2 \sigma _5}{3}+\frac{4}{3} & \frac{\sigma _{-10}}{3}-\sigma _{-2}+\frac{\sigma _5}{3}-1 & -\frac{\sigma _{-10}}{3}+\frac{\sigma _{-2}}{3}-\frac{\sigma _5}{3}+\frac{1}{3} & \frac{\sigma _{-10}}{3}-\sigma _{-2}-\frac{2 \sigma _5}{3} & \frac{2 \sigma _{-2}}{3}-\frac{1}{3} & 1 & -\frac{\sigma _5}{2}-\frac{1}{2} & 1 & \frac{\sigma _5}{2}+\frac{1}{2} \\
				0 & 0 & 0 & 0 & 1 & \frac{\sigma _{-10}}{3}+\frac{\sigma _{-2}}{3}+\frac{2 \sigma _5}{3}-\frac{4}{3} & -\frac{\sigma _{-10}}{3}+\sigma _{-2}+\frac{\sigma _5}{3}-1 & \frac{\sigma _{-10}}{3}-\frac{\sigma _{-2}}{3}-\frac{\sigma _5}{3}+\frac{1}{3} & \frac{\sigma _{-10}}{3}-\sigma _{-2}+\frac{2 \sigma _5}{3} & -\frac{1}{3} 2 \sigma _{-2}-\frac{1}{3} & 1 & -\frac{\sigma _5}{2}-\frac{1}{2} & -1 & \frac{\sigma _5}{2}+\frac{1}{2} \\
				0 & 0 & 0 & 0 & 1 & \frac{\sigma _{-10}}{3}+\frac{\sigma _{-2}}{3}+\frac{2 \sigma _5}{3}-\frac{4}{3} & \frac{\sigma _{-10}}{3}-\sigma _{-2}-\frac{\sigma _5}{3}+1 & \frac{\sigma _{-10}}{3}-\frac{\sigma _{-2}}{3}-\frac{\sigma _5}{3}+\frac{1}{3} & -\frac{\sigma _{-10}}{3}+\sigma _{-2}-\frac{2 \sigma _5}{3} & -\frac{1}{3} 2 \sigma _{-2}-\frac{1}{3} & 1 & \frac{\sigma _5}{2}+\frac{1}{2} & 1 & \frac{\sigma _5}{2}+\frac{1}{2} \\
				1 & 0 & 0 & -2 s_2 & 0 & 0 & 0 & \frac{3 \sigma _5}{2}+\frac{3}{2} & 0 & 1 & 0 & 1 & -\frac{i}{2}-\frac{\sigma _{-5}}{2} & 0 \\
				1 & 0 & 0 & 2 s_2 & 0 & 0 & 0 & \frac{3 \sigma _5}{2}+\frac{3}{2} & 0 & 1 & 0 & 1 & \frac{i}{2}+\frac{\sigma _{-5}}{2} & 0 \\
				1 & -2 s_2 & 0 & 0 & 1 & 0 & 0 & 0 & 0 & \frac{3}{2}-\frac{3 \sigma _5}{2} & 0 & 0 & 1 & \frac{i}{2}-\frac{\sigma _{-5}}{2} \\
				1 & 2 s_2 & 0 & 0 & 1 & 0 & 0 & 0 & 0 & \frac{3}{2}-\frac{3 \sigma _5}{2} & 0 & 0 & 1 & \frac{\sigma _{-5}}{2}-\frac{i}{2} \\
			\end{array}\]
			\caption{Equations of $16$ conics which form a Kummer configuration on $X$.}\label{fig:kum_bas}
		\end{figure}
\end{landscape}}

\providecommand{\bysame}{\leavevmode\hbox to3em{\hrulefill}\thinspace}
\providecommand{\MR}{\relax\ifhmode\unskip\space\fi MR }
\providecommand{\MRhref}[2]{%
	\href{http://www.ams.org/mathscinet-getitem?mr=#1}{#2}
}
\providecommand{\href}[2]{#2}


\begin{thebibliography}{10}
	\bibitem{BH}
	S.~Brandhorst, K.~Hashimoto, \emph{ Extensions of maximal symplectic actions on K3 surfaces},
	\newblock Annales Henri Lebesgue. \textbf{4} (2021) 785--809.
	
	\bibitem{BS}
	C.~Bonnafé. A.~Sarti, \emph{K3 surfaces with maximal finite automorphism groups containing $M_{20}$},
	\newblock Annales de l'Institut Fourier, \textbf{71} (2021), no.~2, 711--730. 
	
	\bibitem{magma}
	W.~Bosma, J.~Cannon, and C.~Playoust, \emph{The {M}agma algebra system. {I}.
		{T}he user language}, J. Symbolic Comput. \textbf{24} (1997), no.~3-4,
	235--265, Computational algebra and number theory (London, 1993).
	
	\bibitem{Alex1}
	A.~Degtyarev, \emph{800 conics in a smooth quartic surface},
	\newblock Journal of Pure and Applied Algebra. \textbf{226} (2022), no. 10, 1--5.
	\bibitem{Alex2}
	A.~Degtyarev, \emph{Conics in Kummer quartics},
	\newblock arXiv:2108.11181v1, 2021.
	
	\bibitem{Deland}
	M.~F.~DeLand, \emph{Geometry of Rational Curves on Algebraic Varieties},
	\newblock PhD thesis, 2009.
	
	\bibitem{ST}
	G.~C. Shephard and J.~A. Todd, \emph{Finite unitary reflection groups},
	Canadian J. Math. \textbf{6} (1954), 274--304.
	
\end{thebibliography}
\end{document}